\documentclass[12pt]{article}
\usepackage[utf8]{inputenc}
\usepackage[T1]{fontenc}
\usepackage[numbers,sort&compress]{natbib}
\usepackage[margin=1in]{geometry}
\usepackage{amsmath, amssymb, amsthm}
\usepackage{graphicx}
\usepackage{hyperref}
\usepackage{natbib}
\usepackage{booktabs}
\usepackage{array}
\usepackage{xcolor}
\usepackage{setspace}
\usepackage{caption}
\usepackage{subcaption}
\usepackage{mathtools}
\usepackage{lineno}

\hypersetup{colorlinks=true,linkcolor=blue!60!black,citecolor=blue!60!black,urlcolor=blue!60!black}
\newtheorem{theorem}{Theorem}[section]

\newtheorem{corollary}[theorem]{Corollary}
\newtheorem{lemma}[theorem]{Lemma}
\newtheorem{remark}{Remark}[section]

\newcommand{\st}{\tilde{s}}
\newcommand{\Cs}{C^{s}}
\newcommand{\Ca}{C^{a}}
\newcommand{\Cr}{C^{r}}
\newcommand{\Cso}{C^{s*}}   
\newcommand{\Cao}{C^{a*}}   
\newcommand{\Csa}{C^{s\ast}} 
\newcommand{\Rz}{\mathcal{R}_0}
\newcommand{\Rphi}{\mathcal{R}_\phi}
\newcommand{\Rmu}{\mathcal{R}_\mu}
\newcommand{\Rc}{\mathcal{R}_c}
\newcommand{\Rcb}{\mathcal{R}_c^{\,b}}

\newcommand{\bb}{\beta_{\mathrm{bio}}}
\newcommand{\Rcl}{R_{\mathrm{cl}}}

\begin{document}

\title{When Self-Protection Backfires: Adaptive Contact Behavior Expands the Endemic Basin in an Addiction Model with Nonlinear Relapse}
\author{%
Fabio Sanchez\and
  Escuela de Matemática-CIMPA, Universidad de Costa Rica,\\ Ciudad Universitaria Rodrigo Facio,
  San José, 11501, Costa Rica.
}
\date{\today}
\maketitle

\begin{abstract}
We extend the Susceptible--Addicted--Reformed (SAR) model of \cite{sanchez2023}, which exhibits a forward--backward bifurcation driven by nonlinear relapse, by embedding an epi-economic behavioral layer in the spirit of \cite{fenichel2011}. At-risk individuals choose contact levels by solving a finite-horizon dynamic program that balances the utility of social activity against the expected cost of addiction. We prove that the basic reproduction number \( R_0 \) and the local stability of the addiction-free equilibrium remain unchanged by the behavioral layer. However, the endemic structure is fundamentally altered: the behavioral response collapses exactly to a scalar mixing factor \( M \), and the bifurcation curve factorizes as \( R_0(a) = R_{\rm cl}(a)/M(a) \). This yields an exact comparison principle: the saddle-node fold shifts to lower \( R_0 \) (enlarging the endemic basin) if and only if \( M \ge 1 \) along the branch. For rational self-protective behavior under conditional proportional mixing, we prove \( M \ge 1 \), so the basin enlarges; the opposite holds under frequency-dependent mixing. Numerical continuation shows that, at baseline parameters, the fold moves left by \( \Delta R_0 \approx -0.035 \) and the critical initial addiction level drops by 2--6 percentage points. Gillespie simulations confirm that the enlarged basin increases the stochastic probability of addiction establishment by up to threefold near threshold. This counterintuitive result that rational self-protection can make addiction easier to establish has direct implications for prevention policy.

\end{abstract}



\section{Introduction}
\label{sec:intro}

Addiction is among the most persistent of global public-health challenges. The World Health Organization and the United Nations Office on Drugs and Crime document substance-use disorders affecting hundreds of millions of people worldwide, with a burden that has continued to grow rather than recede \citep{WHO2024,UNODC2025}; recent analyses of the Global Burden of Disease data find the disability and mortality attributable to drug-use disorders still rising, and project further increases for decades to come \citep{Dong2025}. A defining feature of this burden is its \emph{persistence}: addictive disorders relapse, recur, and re-establish in populations even under sustained intervention.
The policies meant to contain them prevention campaigns, treatment subsidies, and deterrence programs are increasingly evaluated through mathematical models that treat addiction as a socially transmitted condition spreading through contact \citep{sanchez2007,sanchez2023}. This epidemic framing of addictive behavior has a long lineage, spanning compartmental models of cigarette smoking \citep{Rowe1992}, optimal-control formulations of drug epidemics \citep{Behrens2000}, and SIR-type models of alcohol use \citep{sanchez2007}, in which initiation, cessation, and relapse are cast as transitions driven by social contact. Among these, the Susceptible--Addicted--Reformed (SAR) model of \citep{sanchez2023} is notable for its analytic richness: a nonlinear relapse mechanism, modulated by a reducing function $g(\st)$ that captures the deterrent influence of reformed individuals, produces a \emph{forward--backward bifurcation}. The system can sustain addiction even when the basic reproductive number $\Rz<1$, provided the initial addicted population is large enough. Outcome then depends on initial conditions, a hallmark of backward bifurcation and hysteresis \citep{DushoffHuangCastilloChavez1998}, and a feature with sharp policy consequences: driving $\Rz$ below one need not eradicate an established epidemic.

Yet in this and most epidemic-style models of addiction, the contact process that drives transmission is held fixed. This is a conspicuous omission for addiction in particular, where the central object of an individual's exposure to risky social settings is itself a behavioral choice. A separate literature, originating with \citep{fenichel2011}, makes precisely this point for infectious disease: the contact rate is not a biological constant but a decision that responds to perceived risk. Modeling susceptibles as forward-looking agents who reduce contact as prevalence rises, Fenichel et al.\ showed that ignoring adaptation biases epidemiological inference and reshapes an epidemic's trajectory. The broader behavioral-epidemiology literature has since developed this theme extensively \citep[e.g.][]{funk2010,bauch2013}, including a game-theoretic strand in which individuals trade off the costs of precaution against perceived infection risk \citep{reluga2010}, but almost entirely for forward-bifurcating, SIR-type systems. Backward bifurcations are themselves well documented in addiction models, where they arise from structural mechanisms such as multi-strain superinfection or saturated treatment \citep{Memarbashi2022,sanchez2023}; what has not been examined
is how \emph{adaptive behavior} interacts with such a bifurcation, that is, whether endogenous, forward-looking contact choice enlarges or contracts the bistable region in which addiction persists below threshold. That intersection of behavioral economics and backward bifurcation dynamics is the gap this paper addresses.

We close that gap by embedding the Fenichel behavioral mechanism in the Sanchez SAR model, and the result overturns a natural intuition. One would expect that allowing at-risk individuals to protect themselves and rationally withdraw from risky contact as addiction spreads would suppress the epidemic. We prove the opposite: under standard proportional mixing, self-protective behavior \emph{enlarges} the basin of attraction of the endemic state, shifting the backward bifurcation to lower $\Rz$ and making established addiction both easier to ignite and harder to eradicate. The mechanism is not a numerical artifact but a structural one: when susceptibles cut contact, they shrink the shared contact pool
faster than they reduce their exposure to addicted individuals, so each remaining contact carries an elevated risk.

Our analysis makes this precise and, crucially, general. We show that the entire behavioral layer collapses to a single scalar modifier $M$ on the force of infection, and that along the endemic branch the bifurcation curve factorizes as $\Rz(a)=\Rcl(a)/M(a)$, where $\Rcl$ is the classical Sanchez curve. This yields an exact comparison principle: the saddle-node fold moves to lower $\Rz$ (enlarging the basin) if and only if $M\ge1$ along the branch, and to higher $\Rz$ if $M\le1$. The direction of the effect thus hinges on a single inequality and is \emph{independent of the functional form} of the contact coupling, resolving what might otherwise look like a modeling-choice artifact into a sharp dichotomy. We prove $M\ge1$ for self-protective behavior under conditional proportional mixing (so the basin enlarges), whereas a frequency-dependent coupling gives $M\le1$ (so it shrinks); the two predictions bracket the classical fold and are, in principle, empirically distinguishable.

The remainder of the paper develops these results. Sections~\ref{sec:model}--\ref{sec:behavior} construct the epi-economic SAR model, in which the force of addiction depends on an endogenous contact level chosen via dynamic programming and calibrated to nest the classical system. Section~\ref{sec:analysis} proves that $\Rz$ and the stability of the addiction-free equilibrium are invariant to the behavioral layer (Theorem~\ref{thm:R0}) and establishes the mixing-factor collapse (Lemma~\ref{lem:Mfactor}). Section~\ref{sec:general} proves the comparison principle. Section~\ref{sec:numerical} quantifies the effect through a parameter sweep, bifurcation continuation, and Gillespie simulation, showing the impact concentrates near $\Rz=1$ and raises stochastic establishment probability up to threefold near threshold. Section~\ref{sec:mechanism} traces the mechanism, Section~\ref{sec:policy} draws out policy implications, and Section~\ref{sec:discussion} concludes with limitations.

\section{The Extended Model}
\label{sec:model}

\subsection{Classical SAR dynamics}
\label{subsec:classic}

We work with the scaled Susceptible--Addicted--Reformed model of
\citep{sanchez2023}. Let $s,a,\st\in[0,1]$ denote the fractions of the population that are susceptible (at risk but not addicted), addicted, and reformed (formerly addicted, currently abstaining), respectively. Because births and deaths balance under a constant turnover rate $\mu$, the fractions satisfy $s+a+\st=1$ for all time, and the dynamics evolve on the simplex
\[
\Delta=\{(s,a,\st):s,a,\st\ge0,\ s+a+\st=1\}.
\]
The model reads
\begin{equation}
\label{eq:sar_classic}
\dot{s}=\mu-\beta g(\st)\,s\,a-\mu s,\quad
\dot{a}=\beta g(\st)\,s\,a+\phi\,\st\,a-(\mu+\gamma)a,\quad
\dot{\st}=\gamma a-\phi\,\st\,a-\mu\st,
\end{equation}
with the \emph{reducing function}
\begin{equation}
\label{eq:reducing}
g(\st)=\frac{\kappa}{1+\nu\st},\qquad \kappa\in(0,1],\ \nu\in[0,1].
\end{equation}

Each term has a behavioral reading. New addiction arises through social contact between susceptibles and the addicted at rate $\beta g(\st)\,s\,a$: the parameter $\beta$ scales the strength of social influence, while the prevalence-dependent factor $g(\st)$ encodes the deterrent effect of the reformed population. As the reformed fraction $\st$ grows, $g$ decreases; reformed individuals discourage the at-risk from initiating, and the parameter $\nu$ tunes how strongly this deterrence operates, with $\nu=0$ recovering a constant transmission coefficient $\kappa$. The term $\phi\,\st\,a$ is the model's defining feature: a \emph{nonlinear relapse} in which contact between reformed and currently addicted individuals drives the reformed back into addiction at rate $\phi$. Recovery ($a\to\st$) occurs at rate $\gamma$, and $\mu$ governs the demographic turnover by which individuals enter the susceptible class and leave every class.

It is the nonlinear relapse term that makes the model analytically rich. The basic reproductive number, obtained from the addiction-free equilibrium
$(s,a,\st)=(1,0,0)$, is
\begin{equation}
\label{eq:thresholds}
\Rz=\frac{\beta\kappa}{\mu+\gamma},\qquad
\Rphi=\frac{\phi}{\mu+\gamma},\qquad
\Rmu=\frac{\mu}{\mu+\gamma},
\end{equation}
where $\Rz$ measures recruitment per addicted individual, $\Rphi$ the relapse-driven secondary recruitment, and $\Rmu$ the demographic dilution. \citep{sanchez2023} show that when $\Rphi>1$, the relapse feedback is strong enough to produce a \emph{forward--backward bifurcation}: a stable endemic state can coexist with the stable addiction-free equilibrium for a range of $\Rz<1$, separated by an unstable branch, so that the eventual outcome depends on the initial addicted fraction. The saddle-node fold terminating this bistable region occurs at a critical value $\Rc<1$ given in closed form \citep[Eq.~5]{sanchez2023} when $\nu=0$, and as the root of a cubic \citep[Eq.~6]{sanchez2023} when $\nu>0$. This fold, and how behavior moves it, is the central object of our analysis.

\subsection{Making the contact process a decision}
\label{subsec:contact}

In \eqref{eq:sar_classic}, the rate of new addiction is proportional to the product $s\,a$, so the contact process is fixed: every individual mixes at the same implicit rate, absorbed into $\beta$. For an addictive behavior, this is a strong assumption, because exposure to the relevant social settings, peer groups, venues, and occasions where initiation occurs is precisely what an at-risk individual can choose to moderate. We therefore make the contact process explicit, following the epi-economic formulation of \citep{fenichel2011}, which belongs to a longer tradition of treating epidemic exposure as a rational economic choice
\citep{GeoffardPhilipson1996}.

Each health type $h\in\{s,a,r\}$ selects a contact level $C^h\ge0$, interpreted as the intensity of participation in risky social activity. Mixing is \emph{conditional} (assortative through activity level) and \emph{proportional}: the rate of type-$m$--type-$n$ contacts is
\begin{equation}
\label{eq:mixing}
C^{mn}=\frac{C^m C^n}{\bar C},\qquad
\bar C:=s\Cs+a\Ca+\st\Cr,
\end{equation}
where $\bar C$ is the population-mean contact level. This is the standard proportionate-mixing kernel \citep{Jacquez1988,BlytheCastilloChavez1989,BusenbergCastilloChavez1991},
adopted in the epi-economic setting by \citep{fenichel2011}: the probability that a contact made by a type-$m$ individual is with a type-$n$ individual is the share $nC^n/\bar C$ of total contact activity supplied by type $n$, so partners are encountered in proportion to their contact effort rather than their raw frequency. Holding all contact levels equal recovers frequency-dependent transmission, whereas heterogeneous levels make a less-active group correspondingly less likely to encounter the feature that proves decisive once addicted individuals choose lower contact than susceptibles.

The force of addiction on a susceptible who chooses contact level $\Cs$ is the transmission rate per susceptible: the biological susceptibility $\bb$, times the deterrence-modulated prevalence $g(\st)\,a$, times the conditional susceptible--addicted contact rate $\Cs\Ca/\bar C$,
\begin{equation}
\label{eq:foi}
\lambda(\Cs;s,a,\st)=\bb\,g(\st)\,a\,\frac{\Cs\Ca}{s\Cs+a\Ca+\st\Cr}.
\end{equation}

\paragraph{Calibration and nesting.}
The new parameter $\bb$ is a purely biological per-contact transmissibility, distinct from the composite social-influence rate $\beta$ of the classical model. We fix it by requiring that the extended model reduce \emph{exactly} to the classical one whenever contacts are held at their behavior-free benchmark. At the addiction-free equilibrium $(s,a,\st)\to(1,0,0)$ all individuals are susceptible and, as we show in Section~\ref{sec:behavior}, the optimal contact level coincides with the static utility optimum $\Cso=b^s/2$. Imposing $\lambda\to\beta\kappa\,a$ as $a\to0^+$ at this benchmark (verified in Lemma~\ref{lem:calib}) gives
\begin{equation}
\label{eq:calib}
\bb=\frac{\beta}{\Cao}=\frac{2\beta}{b^a},
\end{equation}
where $\Cao=b^a/2$ is the addicted static optimum and $b^a$ the corresponding utility parameter (Section~\ref{sec:behavior}). With this calibration, the basic reproductive number is unchanged from \eqref{eq:thresholds}, so any difference between the two models is attributable to behavioral adaptation away from the addiction-free state, not to a reparameterization of transmission.

Addicted and reformed individuals face no forward-looking risk of initiation: an addicted individual is already addicted, and a reformed individual's relapse risk in this model is governed by the contact term $\phi\st a$ rather than by an initiation decision. Both, therefore, choose their contact levels myopically, at the static utility optima
\begin{equation}
\label{eq:static_contacts}
\Ca=\frac{b^a}{2},\qquad \Cr=\frac{b^r}{2}.
\end{equation}
Only the susceptible's contact level $\Csa$ is forward-looking, determined by the dynamic program of Section~\ref{sec:behavior}.

\subsection{The extended system}
\label{subsec:extended}

Substituting the endogenous force of addiction \eqref{eq:foi} into the SAR dynamics yields the epi-economic model studied in the remainder of the paper:
\begin{equation}
\label{eq:sar_adaptive}
\;
\dot{s}=\mu-\lambda(\Csa)\,s-\mu s,\quad 
\dot{a}=\lambda(\Csa)\,s+\phi\,\st\,a-(\mu+\gamma)a,\quad
\dot{\st}=\gamma a-\phi\,\st\,a-\mu\st,
\;
\end{equation}
where $\Csa=\Csa(s,a,\st)$ is the susceptible's optimal contact level, itself a function of the current state through the optimization of Section~\ref{sec:behavior}. The system \eqref{eq:sar_adaptive} is identical in form to the classical model \eqref{eq:sar_classic} except that the constant transmission term $\beta g(\st)\,s\,a$ is replaced by the state-dependent $\lambda(\Csa)\,s$. As we establish in Section~\ref{sec:analysis}, this single substitution leaves the addiction-free equilibrium and $\Rz$ untouched while reshaping the endemic structure, and, via Lemma~\ref{lem:Mfactor}, the entire behavioral modification can be written as one scalar factor multiplying the classical force of infection.

\section{Behavioral Optimization}
\label{sec:behavior}

The susceptible's contact level $\Csa$ in \eqref{eq:sar_adaptive} is the solution of a finite-horizon dynamic program: at each instant, an at-risk individual chooses how much to participate in risky social activity, trading the flow utility of contact against the discounted risk of becoming addicted. This section establishes the program, characterizes its solution, and shows that it nests the classical model as the zero-foresight limit.

\subsection{Preferences and the cost of addiction}

Each health type $h\in\{s,a,r\}$ derives flow utility from contact according to
\begin{equation}
\label{eq:utility}
u^h(C)=\bigl(b^h C-C^2\bigr)^{\gamma_u}-a^h,\qquad \gamma_u\in(0,1],
\end{equation}
following the contact-utility specification of \citep{fenichel2011}. The quadratic kernel $b^hC-C^2$ makes utility single-peaked in contact; there is an interior best level of social activity, neither isolation nor unbounded mixing with
\emph{static optimum}
\begin{equation}
\label{eq:staticopt}
C^{h*}=\arg\max_{C\ge0}u^h(C)=\frac{b^h}{2}.
\end{equation}
The exponent $\gamma_u$ controls the curvature of utility around this peak: a smaller $\gamma_u$ flattens the objective, so that a given health concern induces a larger contact reduction. The constant $a^h$ is a fixed disutility of occupying state $h$, which shifts the level of welfare but not the optimal contact; we set $a^s=a^r=0$ and $a^a>0$, and scale $a^a$ so that $u^a(C^{a*})=0$, i.e.\ the addicted state is normalized to zero flow welfare at its own optimum. The parameter ordering $b^s=b^r\ge b^a\ge0$ encodes two substantive assumptions: susceptibles and reformed
individuals share the same baseline sociability ($b^s=b^r$), while addicted individuals are \emph{less} socially active at their optimum ($b^a\le b^s$, so $C^{a*}\le C^{s*}$), the empirical regularity that addiction narrows the range of ordinary social participation. As shown below, it is precisely this gap $\rho=b^a/b^s<1$ that drives the paper's central effect.

\subsection{The value of being addicted}

An individual's continuation value depends on their health state. The asymmetry of the SAR dynamics induces a corresponding asymmetry in the decision problem: a susceptible actively chooses contact at every step, whereas an addicted individual is passive, simply waiting to recover at the exogenous per-period probability
\begin{equation}
P^r=1-e^{-\gamma}.
\end{equation}
The value of being addicted is therefore computable in closed form. With $k$ periods remaining in the planning horizon, recovery delivers the reformed flow utility $u^r(C^{r*})$ thereafter, and (Appendix~\ref{app:Vi})
\begin{equation}
\label{eq:Vi}
V_k(a)=u^r(C^{r*})\!\left[\frac{1-\delta^{k}}{1-\delta}
-\frac{1-[\delta(1-P^r)]^{k}}{1-\delta(1-P^r)}\right],\qquad V_0(a)=0,
\end{equation}
with discount factor $\delta\in(0,1)$. The bracketed expression is the expected discounted time spent in the reformed (positive-utility) state over the remaining horizon; it is increasing in $k$ and in the recovery rate $\gamma$, so a longer horizon or faster recovery makes addiction less costly relative to remaining susceptible.

\subsection{The susceptible's dynamic program}

A susceptible chooses contact to maximize discounted expected utility over the horizon, anticipating that contact today raises the probability of transitioning into the (lower-value) addicted state. This optimizing formulation places the model within the decision-theoretic strand of behavioral epidemiology, in which individuals choose protective effort in response to prevailing risk \citep{Chen2009,Toxvaerd2020}. The Bellman equation is
\begin{equation}
\label{eq:bellman}
V_t(s)=\max_{\Cs\in[0,\,b^s/2]}\Bigl\{u^s(\Cs)
+\delta\bigl[(1-P^a(\Cs))\,V_{t+1}(s)+P^a(\Cs)\,V_{t+1}(a)\bigr]\Bigr\},
\end{equation}
with the per-period probability of becoming addicted obtained from the force of addiction \eqref{eq:foi},
\begin{equation}
\label{eq:Pa}
P^a(\Cs)=1-\exp\!\bigl(-\lambda(\Cs;s,a,\st)\bigr).
\end{equation}
The maximization is restricted to $\Cs\in[0,b^s/2]$ because contact above the static optimum $b^s/2$ is dominated: it lowers flow utility \emph{and} raises addiction risk, so no susceptible would ever choose it. On this compact interval, the objective is continuous, so a global maximizer exists. The flow utility $u^s$ is concave on $(0,b^s/2)$, while the addiction-risk term $-\delta(V_{t+1}(s)-V_{t+1}(a))\,P^a(\Cs)$ is convex, since $P^a=1-e^{-\lambda}$ is increasing and concave in $\Cs$ (the force of addiction \eqref{eq:foi} saturates in its own contact). The objective is therefore not globally concave in general, and we do not rely on it being so: we compute the global maximizer directly by grid search on $[0,b^s/2]$ (Appendix~\ref{app:numerics}). In practice, across the entire endemic branch and over all behavioral parameters we consider ($\tau\in[6,60]$ days, $\gamma_u\in[0.1,1]$, $\Rz\in[0.90,1.00]$), the objective is single-peaked at every state, with the optimum $\Csa$ interior whenever addiction risk is positive; we found no instance of multimodality.

We close the program with \emph{adaptive expectations}: the agent treats the current population state $(s,a,\st)$ as persisting over the planning horizon $\tau$, rather than forecasting the full epidemic trajectory. This is the assumption of \citep{fenichel2011} and is the natural baseline for two reasons. First, it is informationally minimal; it requires the agent to perceive only current prevalence, not to solve the population-level dynamics, which is the realistic cognitive posture for individual risk decisions. Second, it yields a well-posed period-by-period optimization that can be embedded in the ODE system without a
fixed-point coupling between individual forecasts and aggregate dynamics. A fully rational-expectations (Hamilton--Jacobi--Bellman) treatment, in which agents anticipate the epidemic trajectory and best-respond to it, is a natural extension \citep[as in the differential-game and equilibrium formulations of][]{reluga2010,Toxvaerd2020}; because forward-looking agents would cut contact \emph{earlier} in a rising epidemic, it would strengthen the precautionary response, though as we discuss in Section~\ref{sec:discussion} a larger reduction acts to \emph{attenuate} rather than amplify the basin effect.

\subsection{Optimality condition and the classical limit}

Differentiating \eqref{eq:bellman}, the optimal contact level satisfies the first-order condition
\begin{equation}
\label{eq:foc}
\underbrace{\frac{\partial u^s}{\partial\Cs}}_{\text{marginal utility of contact}}
=\underbrace{\delta\bigl(V_{t+1}(s)-V_{t+1}(a)\bigr)
\frac{\partial P^a}{\partial\Cs}}_{\text{marginal expected cost of addiction}} .
\end{equation}
The left side is the flow benefit of an additional unit of contact; the right side is its expected cost, the product of the discounted value lost on becoming addicted, $V_{t+1}(s)-V_{t+1}(a)\ge0$, and the marginal increase in addiction risk $\partial P^a/\partial\Cs>0$. The susceptible reduces contact below the static optimum exactly to the point where these balance. Higher addiction prevalence $a$ raises the marginal risk $\partial P^a/\partial\Cs$, tightening the FOC and lowering $\Csa$ over the relevant range; the precise sense in which behavior is ``self-protective'' though, as noted, we obtain $\Csa$ by direct maximization
rather than from the FOC alone.

\begin{remark}[Nesting the classical model]
\label{rem:tau0}
The right-hand side of \eqref{eq:foc} is proportional to the value gap $V_{t+1}(s)-V_{t+1}(a)$, which vanishes as the horizon $\tau\to0$ (equivalently $k\to0$ in \eqref{eq:Vi}, giving $V_0(a)=0$ and $V_0(s)=0$). The marginal cost of contact then drops out, the FOC reduces to $\partial u^s/\partial\Cs=0$, and $\Csa=C^{s*}=b^s/2$: the susceptible exercises no precautionary reduction. Substituting into \eqref{eq:sar_adaptive} recovers the classical force of infection at the onset of addiction (where $M\to1$ as $a\to0^+$, Lemma~\ref{lem:Mfactor}); away from onset the static-contact model retains the proportional-mixing premium $M_{\text{static}}\ge1$ of Eq.~\eqref{eq:Mstatic}, so the $\tau=0$ limit is the \emph{behavior-free} proportional-mixing model rather than the frequency-dependent classical one (Section~\ref{sec:mechanism}). The planning horizon $\tau$ thus interpolates between this static-contact benchmark ($\tau=0$) and the fully adaptive model, providing a single dial that measures the strength of the behavioral response, while $\Rz$ and the addiction-free dynamics remain those of the classical system at every $\tau$.
\end{remark}

Because $P^a$ is nonlinear in $\Cs$ through \eqref{eq:foi}, the program
\eqref{eq:bellman} admits no closed-form solution and is solved numerically by backward induction over $[0,\tau]$ on a grid in $\Cs$, with the state advanced and the program re-solved each period (a rolling horizon); details are in Appendix~\ref{app:numerics}.

\section{Equilibrium Analysis}
\label{sec:analysis}

We now establish two facts that frame the rest of the paper: the behavioral layer leaves the basic reproductive number and the stability of the addiction-free equilibrium (AFE) untouched, yet it modifies the endemic dynamics through a single scalar factor on the force of infection. The first fact states that the behavioral layer is invisible \emph{before} an outbreak; the second specifies exactly how it acts \emph{after} one.

\subsection{Invariance of \texorpdfstring{$\Rz$}{R0}}

The calibration \eqref{eq:calib} was chosen so that the extended model coincides with the classical one at the onset of addiction. We first record this formally.

\begin{lemma}
\label{lem:calib}
Under calibration \eqref{eq:calib}, the force of addiction at the addiction-free equilibrium satisfies $\lim_{a\to0^+}\lambda(C^{s*};1,a,0)/a=\beta\kappa$.
\end{lemma}
\begin{proof}
With $s=1$, $\st=0$, and contacts at their static optima, \eqref{eq:foi} gives $\lambda=\bb\,\kappa\,a\,C^{s*}\Ca/(C^{s*}+a\Ca)$. Dividing by $a$ and letting $a\to0^+$ leaves $\bb\,\kappa\,\Ca$, which equals $\beta\kappa$ by the calibration $\bb=\beta/\Ca$.
\end{proof}

\begin{theorem}
\label{thm:R0}
The addiction-free equilibrium $(1,0,0)$ of \eqref{eq:sar_adaptive} is locally asymptotically stable if $\Rz=\beta\kappa/(\mu+\gamma)<1$ and unstable if $\Rz>1$. The threshold value $\Rz$ is independent of all behavioral parameters $(\tau,\delta,b^s,b^a,\gamma_u)$.
\end{theorem}
\begin{proof}
The key point is that the susceptible's optimal contact deviates from the static optimum only at order $a$. By the first-order condition \eqref{eq:foc}, the contact reduction $C^{s*}-\Csa$ is proportional to the marginal expected cost of addiction, which carries the factor $\partial P^a/\partial\Cs=O(\lambda)=O(a)$ as $a\to0^+$; hence $\Csa=C^{s*}+O(a)$. Since $\lambda$ in \eqref{eq:foi} is itself $O(a)$, the correction to $\Csa$ contributes only at $O(a^2)$ to the force of
infection, and so does not enter the linearization at the AFE.

Linearizing \eqref{eq:sar_adaptive} at $(1,0,0)$ with $\Csa=C^{s*}$ therefore gives, by Lemma~\ref{lem:calib}, the same Jacobian as the classical system:
\[
J(1,0,0)=
\begin{pmatrix}
-\mu & -\beta\kappa & 0\\[2pt]
0 & \beta\kappa-(\mu+\gamma) & 0\\[2pt]
0 & \gamma & -\mu
\end{pmatrix},
\]
with eigenvalues $-\mu$ (multiplicity two) and $\beta\kappa-(\mu+\gamma)=(\mu+\gamma)(\Rz-1)$. The AFE is stable precisely when the last eigenvalue is negative, i.e.\ $\Rz<1$, and the behavioral parameters enter only through the $O(a^2)$ correction, leaving both the Jacobian and $\Rz$ unchanged.
\end{proof}

Theorem~\ref{thm:R0} mirrors the central message of \citep{fenichel2011}: $\Rz$ is a property of the pre-outbreak state and is blind to adaptation, which switches on only once prevalence is positive. The entire effect of behavior is thus an \emph{endemic-regime} phenomenon, which the next subsection localizes to a single multiplicative factor.

\subsection{The behavioral layer collapses to a single mixing factor}

The apparent complexity of the behavioral coupling, an endogenous, state-dependent contact level entering a nonlinear mixing kernel, reduces exactly to one scalar multiplying the classical force of infection.

\begin{lemma}[Mixing-factor collapse]
\label{lem:Mfactor}
Under calibration \eqref{eq:calib}, the force of addiction \eqref{eq:foi} factors as
\begin{equation}
\label{eq:Mcollapse}
\lambda(\Cs;s,a,\st)=\underbrace{\beta\,g(\st)\,a}_{\text{classical force}}\;\cdot\;
M,\qquad
M:=\frac{\Cs}{s\Cs+a\Ca+\st\Cr}=\frac{\Cs}{\bar C},
\end{equation}
where $\bar C=s\Cs+a\Ca+\st\Cr$ is the population-mean contact level, so $M$ is the ratio of the susceptible's own contact rate to the population mean. Moreover $M=1$ at the addiction-free equilibrium, so \eqref{eq:sar_adaptive} nests \eqref{eq:sar_classic} exactly.
\end{lemma}
\begin{proof}
Substituting $\bb=\beta/\Ca$ from \eqref{eq:calib} into \eqref{eq:foi} cancels the factor $\Ca$, leaving $\beta g(\st)a\cdot\Cs/(s\Cs+a\Ca+\st\Cr)$, which is \eqref{eq:Mcollapse}. At the AFE $(s,a,\st)\to(1,0,0)$, so $\bar C\to\Cs$ and $M\to1$.
\end{proof}

Lemma~\ref{lem:Mfactor} is the organizing device of the paper: \emph{every} behavioral effect on the dynamics is carried by the single scalar field $M$, and the entire question of how adaptation reshapes the bifurcation reduces to how $M$ departs from unity along the endemic branch. Two structural facts about $M$ set up that analysis.

First, consider contacts held at their static optima, $\Cs=\Cr=c_0:=b^s/2$ and $\Ca=\rho c_0$ with $\rho=b^a/b^s<1$. Because reformed individuals share the susceptible contact level ($\Cr=\Cs=c_0$), their contribution to the numerator and denominator of $M$ is identical and the $\st$-dependence cancels exactly, leaving
\begin{equation}
\label{eq:Mstatic}
M_{\text{static}}(a)=\frac{c_0}{(1-a-\st)c_0+a\rho c_0+\st c_0}
=\frac{1}{1-a(1-\rho)}=1+(1-\rho)\,a+O(a^2)\ \ge\ 1 .
\end{equation}
The reformed class is therefore ``invisible'' to the mixing factor under static contacts; the deviation of $M$ above unity is driven entirely by the \emph{addicted} class, which is less socially active ($\rho<1$). A susceptible's \emph{relative} exposure rises as the population shifts into this low-activity class, the seed of the basin-enlargement result. Second, $M$ is increasing in the susceptible's own contact,
\[
\frac{\partial M}{\partial\Cs}=\frac{a\Ca+\st\Cr}{\bar C^{2}}>0 ,
\]
so behavioral contact reduction ($\Csa<c_0$) \emph{lowers} $M$, working against the static mixing effect \eqref{eq:Mstatic}. The net position of $M$ relative to unity is thus a contest between the addicted-class mixing effect, which pushes $M>1$, and the susceptible's own precautionary reduction, which pushes $M<1$. Section~\ref{sec:general} resolves this contest and, more importantly, shows that the \emph{direction} of the fold shift is governed entirely by the sign of $M-1$ along the branch, for any behavioral mechanism whatsoever.

\section{A General Comparison Principle for the Fold}
\label{sec:general}

This section contains the paper's main analytical contribution. We show that the behavioral layer enters the bifurcation structure through a single channel, the modifier $M$ of Lemma~\ref{lem:Mfactor}, and that the direction in which the saddle-node fold moves is determined exactly by whether $M\ge1$ or $M\le1$ along the endemic branch, independent of the functional form of the contact coupling. The argument is a comparison principle and is therefore exact: it makes no appeal to the size of the behavioral effect or to perturbation theory.

\subsection{Branch decomposition}

At an endemic equilibrium, setting $\dot\st=0$ in \eqref{eq:sar_adaptive} and using $s=1-a-\st$ expresses the reformed and susceptible fractions as explicit, \emph{$\Rz$-independent} functions of the addicted fraction,
\begin{equation}
\label{eq:branchstate}
\st^*(a)=\frac{\gamma a}{\phi a+\mu},\qquad s^*(a)=1-a-\st^*(a),
\end{equation}
valid for $a\in(0,\bar a)$, where $\bar a$ is the unique root of $s^*(a)=0$ (for the baseline parameters $\bar a\approx0.431$). Dividing the equilibrium $a$-equation by $a$ and using the mixing-factor collapse
(Lemma~\ref{lem:Mfactor}, $\lambda=\beta g(\st)\,a\,M$) gives the endemic balance
\begin{equation}
\label{eq:balance}
\beta\, g(\st^*)\, M\, s^* + \phi\,\st^* = \mu+\gamma .
\end{equation}
Solving \eqref{eq:balance} for $\Rz=\beta\kappa/(\mu+\gamma)$ yields the central identity of the paper.

\begin{lemma}[Branch decomposition]
\label{lem:branch}
Along the endemic branch, the bifurcation curve in the $(\Rz,a)$ plane factorizes as
\begin{equation}
\label{eq:Rcurve}
\Rz(a;M)=\frac{\Rcl(a)}{M(a)},\qquad
\Rcl(a)=\frac{\bigl(1+\nu\st^*(a)\bigr)\bigl[(\mu+\gamma)-\phi\st^*(a)\bigr]}
{(\mu+\gamma)\,s^*(a)},
\end{equation}
where $\Rcl$ is exactly the classical ($M\equiv1$) Sanchez curve and $M(a)$ denotes the modifier evaluated along the branch \eqref{eq:branchstate}. The behavioral layer enters \eqref{eq:Rcurve} only as the divisor $M(a)$.
\end{lemma}
\begin{proof}
Equation~\eqref{eq:balance} gives $\beta g(\st^*)M s^* = (\mu+\gamma)-\phi\st^*$; with $g(\st^*)=\kappa/(1+\nu\st^*)$ this rearranges to $\beta\kappa = (1+\nu\st^*)\bigl[(\mu+\gamma)-\phi\st^*\bigr]/(M s^*)$. Dividing by $\mu+\gamma$ gives \eqref{eq:Rcurve}. The factorization is exact because $\st^*$ and $s^*$ in \eqref{eq:branchstate} do not depend on $\Rz$.
\end{proof}

The saddle-node fold $\Rcb$ is, by definition, the smallest $\Rz$ for which a positive endemic equilibrium exists, the lowest point of the branch curve. By Lemma~\ref{lem:branch} this is
\begin{equation}
\label{eq:foldmin}
\Rcb(M)=\inf_{a\in(0,\bar a)}\frac{\Rcl(a)}{M(a)} .
\end{equation}
We verify the formula against known cases. At $M\equiv1$, the classical curve $\Rcl$ is unimodal on $(0,\bar a)$, and $\min_a\Rcl(a)$ reproduces the closed-form Sanchez fold: numerically $0.64955$ at $\nu=0$, matching \citep[Eq.~5]{sanchez2023} to five digits, and $0.90140$ at $\nu=0.8$. With the adaptive modifier $M(a)$, evaluating \eqref{eq:foldmin} gives $\Rcb=0.866$ at $\nu=0.8$, matching the directly continued bifurcation diagram of Section~\ref{sec:numerical} to three digits.

\subsection{The comparison theorem}

The fold \eqref{eq:foldmin} is an infimum of $\Rcl/M$, so its dependence on the modifier is governed by an elementary monotonicity.

\begin{theorem}[Comparison principle for the fold]
\label{thm:general}
Let $M_1,M_2$ be behavioral modifiers with $M_1(a)\ge M_2(a)>0$ for all $a\in(0,\bar a)$. Then $\Rcb(M_1)\le\Rcb(M_2)$. In particular, taking $M_2\equiv1$ and writing $\Rc^{\mathrm{cl}}=\inf_a\Rcl(a)$,
\begin{equation}
M(a)\ge 1 \text{ on } (0,\bar a)\ \Longrightarrow\ \Rcb(M)\le\Rc^{\mathrm{cl}}
\quad\text{(the fold moves left; the endemic basin enlarges),}
\end{equation}
with strict inequality whenever $M(a)>1$ on a neighborhood of some classical minimizer $a^\star\in\arg\min_a\Rcl(a)$. Symmetrically, $M(a)\le1$ on $(0,\bar a)$ implies $\Rcb(M)\ge\Rc^{\mathrm{cl}}$ (the fold moves right; the basin shrinks).
\end{theorem}
\begin{proof}
Pointwise $M_1\ge M_2>0$ gives $\Rcl(a)/M_1(a)\le\Rcl(a)/M_2(a)$ for every
$a\in(0,\bar a)$, since $\Rcl(a)>0$. Taking the infimum over $a$ preserves the inequality, so $\Rcb(M_1)=\inf_a\Rcl/M_1\le\inf_a\Rcl/M_2=\Rcb(M_2)$. For strictness with $M_2\equiv1$: if $M>1$ on a neighborhood $U$ of a minimizer $a^\star$, then $\Rcl(a^\star)/M(a^\star)<\Rcl(a^\star)=\Rc^{\mathrm{cl}}$, so the infimum of $\Rcl/M$ is strictly below $\Rc^{\mathrm{cl}}$. The reverse statement follows by exchanging the roles of $M_1$ and $M_2$.
\end{proof}

\begin{remark}[Why no unimodality is needed]
\label{rem:nounimodal}
The theorem compares the \emph{infima} of two curves and so requires neither $\Rcl$ nor $\Rcl/M$ to be unimodal: the saddle-node fold is by definition the global minimum of the branch (the least $\Rz$ admitting an endemic state), and the comparison acts on that global minimum directly. Unimodality of $\Rcl$, which we observe numerically for $\Rphi>1$, is convenient for locating $a^\star$ but plays no role in the inequality. (The adaptive curve $\Rcl/M$ may show faint secondary ripples under coarse numerical evaluation of $M$; these are artifacts of the grid-based contact solver and do not affect the global minimum, which is stable
under grid refinement.)
\end{remark}

\noindent The theorem depends on the modifier \emph{only} through the pointwise comparison with unity. The functional form of the mixing is irrelevant to the direction of the shift; all that matters is whether the behavioral mechanism amplifies ($M>1$) or damps ($M<1$) transmission along the branch.

\begin{corollary}[Mixing-form dichotomy]
\label{cor:dichotomy}
Conditional proportional mixing \eqref{eq:mixing} yields $M(a)\ge1$ along the branch (Lemma~\ref{lem:m1pos} below) and hence \emph{enlarges} the endemic basin, $\Rcb<\Rc^{\mathrm{cl}}$. A frequency-dependent coupling
$\lambda=\beta g(\st)\,a\,(\Cs/c_0)$ yields $M(a)=\Cs/c_0\le1$ and hence \emph{shrinks} it, $\Rcb>\Rc^{\mathrm{cl}}$. The two predictions bracket the classical fold, and any behavioral mechanism is classified by which side of $M\equiv1$ its branch profile falls on.
\end{corollary}

\subsection{The mixing factor exceeds unity for self-protective behavior}

It remains to be verified that $M\ge1$ for the adaptive model. From the collapse \eqref{eq:Mcollapse}, with $\Cr=c_0:=b^s/2$ and $\Ca=\rho c_0$,
\begin{equation}
\label{eq:Mge1}
M\ge1 \iff a(\Cs-\Ca)+\st(\Cs-\Cr)\ge0 \iff \Cs\ge C^{s}_{\mathrm{crit}}
:=c_0\,\frac{a\rho+\st}{a+\st},
\end{equation}
an exact equivalence; since $\rho<1$, the threshold satisfies
$C^{s}_{\mathrm{crit}}<c_0$, so $M\ge1$ unless the susceptible cuts contact below $C^{s}_{\mathrm{crit}}$. We bound the contact reduction $\Delta:=c_0-\Csa$ rigorously, with no appeal to linearization.

\begin{lemma}[Rigorous reduction bound]
\label{lem:delta}
Let $\kappa_u:=\min_{\Cs\in[\Csa,c_0]}\bigl(-\partial^2 u^s/\partial(\Cs)^2\bigr)>0$ be the strong-concavity modulus of flow utility on the interval between the adaptive and static optima, let $L_p:=\sup_{\Cs}\partial P^a/\partial\Cs$ (attained at the lower endpoint, since $\partial\lambda/\partial\Cs$ is decreasing in $\Cs$), and let
\[
K:=\delta\bigl(V(s)-V(a)\bigr)\le
\delta\Bigl[u^s(c_0)\,\tfrac{1-\delta^{\tau}}{1-\delta}-V_\tau(a)\Bigr]=:K_{\max},
\]
the last inequality holds because a susceptible earns at most $u^s(c_0)$ per period and $V_\tau(a)$ is the closed form \eqref{eq:Vi}. Then
\begin{equation}
\label{eq:Deltabound}
\Delta \;\le\; \frac{2K_{\max}L_p}{\kappa_u}.
\end{equation}
\end{lemma}
\begin{proof}
Write the per-period objective as $J(\Cs)=u^s(\Cs) K\,P^a(\Cs)+\text{const}$, with $K\ge0$. Optimality of $\Csa$ gives $J(\Csa)\ge J(c_0)$, i.e.
\[
u^s(c_0)-u^s(\Csa)\le K\bigl(P^a(c_0)-P^a(\Csa)\bigr).
\]
Since $\partial u^s/\partial\Cs(c_0)=0$ and $u^s$ is $\kappa_u$-strongly concave on $[\Csa,c_0]$, the left side is at least $\tfrac{\kappa_u}{2}\Delta^2$. Since $P^a$ is $L_p$-Lipschitz, the right side is at most $K L_p\Delta\le K_{\max}L_p\Delta$. Combining and canceling one factor of $\Delta>0$ gives \eqref{eq:Deltabound}.
\end{proof}

\begin{lemma}[Mixing dominates behavior]
\label{lem:m1pos}
$M\ge1$ along the endemic branch whenever $2K_{\max}L_p/\kappa_u\le c_0\,a(1-\rho)/(a+\st)$, i.e.\ whenever the rigorous reduction bound \eqref{eq:Deltabound} does not exceed $c_0-C^{s}_{\mathrm{crit}}$. This sufficient condition holds, with substantial margin, for all planning
horizons $\tau\le24$ days and transmission rates $\beta\le0.045$ a region that contains the entire addiction-relevant regime (baseline $\tau=12$ days, $\beta\kappa\approx0.003\,\mathrm{day}^{-1}$). A direct numerical solution of the first-order condition confirms $M\ge1$ throughout the fold region (addicted fractions $a\lesssim0.25$, which contains the saddle-node minimizer at $a\approx0.13$) over the wider box $\tau\le82$ days and $\beta\le0.052$ ($\Rz\le5.9$); along the branch at baseline the realized margin $\Csa-C^{s}_{\mathrm{crit}}$ ranges from $+0.09$ to $+0.51$ contacts ($M$ from $1.00$ to $1.10$). Far up the stable upper branch ($a\gtrsim0.3$), under long horizons, a strong precautionary reduction can drive $M$ below unity; this region lies well above the fold and so does not affect the basin boundary, on which the comparison principle (Theorem~\ref{thm:general}) acts.
\end{lemma}
\begin{proof}
Combining $\Delta\le2K_{\max}L_p/\kappa_u$ (Lemma~\ref{lem:delta}) with the hypothesis gives $\Delta\le c_0-C^{s}_{\mathrm{crit}}$, hence $\Csa\ge C^{s}_{\mathrm{crit}}$ and $M\ge1$ by \eqref{eq:Mge1}; the inequality is strict for $a>0$ because $\rho<1$. The stated $(\tau,\beta)$ region is the set on which the displayed inequality is verified to hold along the branch (Appendix~\ref{app:numerics}); the wider numerically-verified range relaxes the conservative constants $\kappa_u,L_p,K_{\max}$ to their realized values.
\end{proof}

Combining Theorem~\ref{thm:general} with Lemma~\ref{lem:m1pos} gives the main analytical conclusion: \emph{for self-protective behavior under conditional proportional mixing, the fold moves to lower $\Rz$ and the endemic basin enlarges, throughout the addiction-relevant parameter regime.}

\subsection{Local companion: the transcritical bifurcation at \texorpdfstring{$\Rz=1$}{R0=1}}
\label{subsec:ccs}

Theorem~\ref{thm:general} concerns the global saddle-node fold, which generates the bistability and is the bifurcation of interest here. For completeness, we record a complementary local statement about the \emph{transcritical} bifurcation at $\Rz=1$, obtained from the Castillo--Chavez--Song center-manifold criterion \citep{castillochavez2004}. Writing the modifier's linearization at the AFE as $M=1+M_a\,a+M_{\st}\,\st+O(2)$, the CCS bifurcation coefficient decomposes as
\begin{equation}
\label{eq:ccs}
a_{\mathrm{CCS}}=a_{\mathrm{cl}}
+\frac{2\mu^2(\gamma+\mu)}{\gamma^2}\,m_1,
\qquad m_1:=M_a+M_{\st}\frac{\gamma}{\mu}=\frac{dM}{da}\Big|_{a\to0^+}\ \text{(along the branch)},
\end{equation}
with $b_{\mathrm{CCS}}=\kappa\mu/\gamma>0$, so a backward bifurcation at $\Rz=1$ occurs iff $a_{\mathrm{CCS}}>0$; this is the standard criterion for sub-threshold endemic equilibria \citep{vandenDriesscheWatmough2002,Brauer2004}. The behavioral mechanism enters the transcritical direction also through a single scalar, the branch slope $m_1=M'(0)$, and $m_1>0$ (which $M\ge1$ together with $M(0)=1$ implies) shifts $a_{\mathrm{CCS}}$ toward the backward regime. The local ($\Rz=1$) and global (fold) analyses thus isolate the same control variable: the rate at which the modifier amplifies transmission as prevalence increases. This is precisely the prevalence-amplifying feedback that drives backward bifurcations in the core-group tradition \citep{HethcoteYorke1984,hadeler1995,hethcote2000}, one of several known mechanisms that generate them \citep{Gumel2012}; our contribution is to show that rational self-protective behavior, under proportional mixing, \emph{generates} such a feedback rather than suppressing it.

\section{Numerical Results}
\label{sec:numerical}

This section quantifies the analytical results of Section~\ref{sec:general}. We integrate the adaptive system \eqref{eq:sar_adaptive} with an adaptive LSODA solver, recomputing the optimal contact $\Csa$ at each step by backward induction (Appendix~\ref{app:numerics}). Baseline parameters (Table~\ref{tab:params}) are those of \citep{sanchez2023} in the forward--backward regime ($\nu=0.8$, $\Rphi=1.54$), augmented with the behavioral parameters of \citep{fenichel2011}; $\kappa$ is varied to sweep $\Rz$ while holding the epidemiological ratios fixed. Four results follow: the effect is concentrated near $\Rz=1$ (\S\ref{subsec:where}), the fold shifts left by the amount the comparison principle predicts (\S\ref{subsec:bif}), the deterministic basin enlarges (\S\ref{subsec:basin}), and the enlargement survives stochastically in finite populations (\S\ref{subsec:stochastic}).

\begin{table}[h!]
\centering
\caption{Baseline parameters.}
\label{tab:params}
\small
\begin{tabular}{lll}
\toprule
Symbol & Value & Meaning\\
\midrule
$\beta$ & $0.009$ & social influence (recruitment) \\
$\kappa$ & varies & cost of addiction (sweeps $\Rz$) \\
$\gamma$ & $0.0027$ & recovery rate \\
$\phi$ & $0.0044$ & relapse rate ($\Rphi=1.54$) \\
$\mu$ & $0.00015$ & turnover rate \\
$\nu$ & $0.8$ & willingness factor \\
$b^s=b^r$ & $10.0$ & contact utility (susceptible/reformed) \\
$b^a$ & $6.67$ & contact utility (addicted, $\rho=b^a/b^s=0.667$) \\
$\gamma_u$ & $0.25$ & utility curvature \\
$\delta$ & $0.9999$ & daily discount factor \\
$\tau$ & $12$ & planning horizon (days) \\
\bottomrule
\end{tabular}
\end{table}

\subsection{Where behavior matters}
\label{subsec:where}
Figure~\ref{fig:effectsize} reports the relative change in the endemic
equilibrium, $|a^*_{\mathrm{adap}}-a^*_{\mathrm{class}}|/a^*_{\mathrm{class}}$, as $\Rz$ is varied via $\kappa$. The effect is largest in the threshold neighborhood $\Rz\approx1$, reaching $13$--$16\%$ just above the bifurcation, and decays monotonically as transmission strengthens: at large $\Rz$ it falls below $1\%$, because a fast-spreading epidemic saturates the population regardless of contact behavior, so adaptation barely moves the endemic level. The effect is nearly independent of the utility curvature $\gamma_u$ (the three curves differ by under two percentage points). This is exactly the location predicted by Section~\ref{sec:general}: the behavioral modifier $M$ acts on the \emph{position of the fold}, so its influence on the endemic state is greatest where the branch is most sensitive to $\Rz$ in the bistable neighborhood of the bifurcation.

\begin{figure}[ht!]
\centering
\includegraphics[width=0.72\textwidth]{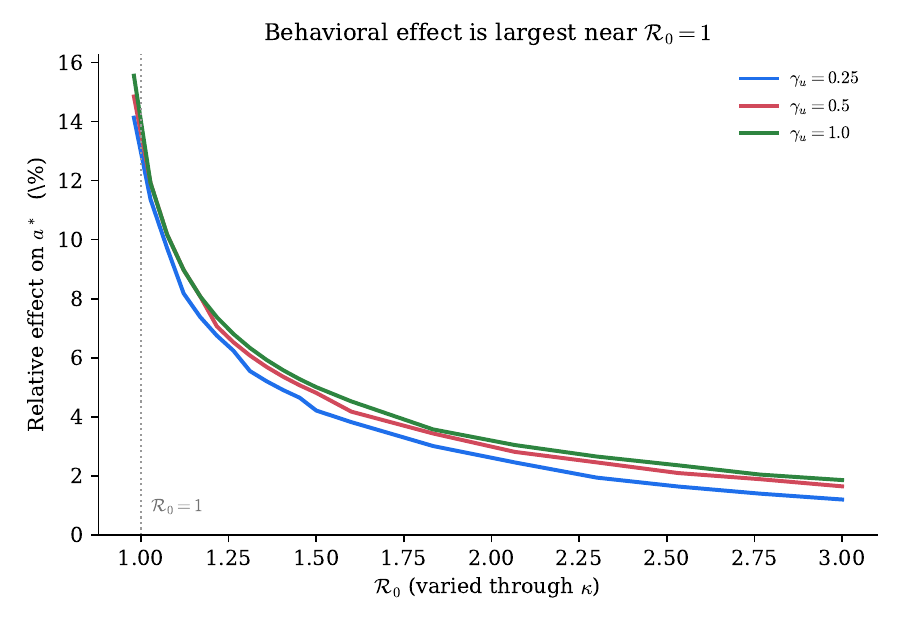}
\caption{Relative effect of adaptive behavior on the endemic equilibrium $a^*$, versus $\Rz$ (varied through $\kappa$), for three utility curvatures. The effect is largest near $\Rz=1$, reaching $13$--$16\%$ just above the bifurcation, and decays monotonically at higher transmission. Seed $a(0)=0.10$; below the classical fold the classical model has no endemic state from this seed, so the curve is shown where the comparison is well defined.}
\label{fig:effectsize}
\end{figure}

\subsection{Bifurcation diagram}
\label{subsec:bif}
Figure~\ref{fig:bifurcation} overlays the bifurcation diagrams of the classical and adaptive systems, computed by multi-initial-condition continuation with the unstable middle branch located by root-finding on the slow-manifold balance \eqref{eq:branchstate} (Appendix~\ref{app:numerics}). Both systems display the backward bifurcation: a stable endemic upper branch, the stable addiction-free branch $a^*=0$, and an unstable middle branch separating their basins. The
adaptive upper branch lies \emph{above} the classical one throughout, and its saddle-node fold occurs at lower $\Rz$: numerically $\Rcb\approx0.866$ against $\Rc^{\mathrm{cl}}\approx0.901$, a shift of $\Delta\Rz\approx-0.035$. This value agrees with the branch-minimum formula \eqref{eq:foldmin} to three significant figures when evaluated independently, confirming both the comparison principle (Theorem~\ref{thm:general}) and the numerical continuation. Adaptive behavior thus expands the bistable region toward smaller $\Rz$.

\begin{figure}[ht!]
\centering
\includegraphics[width=0.72\textwidth]{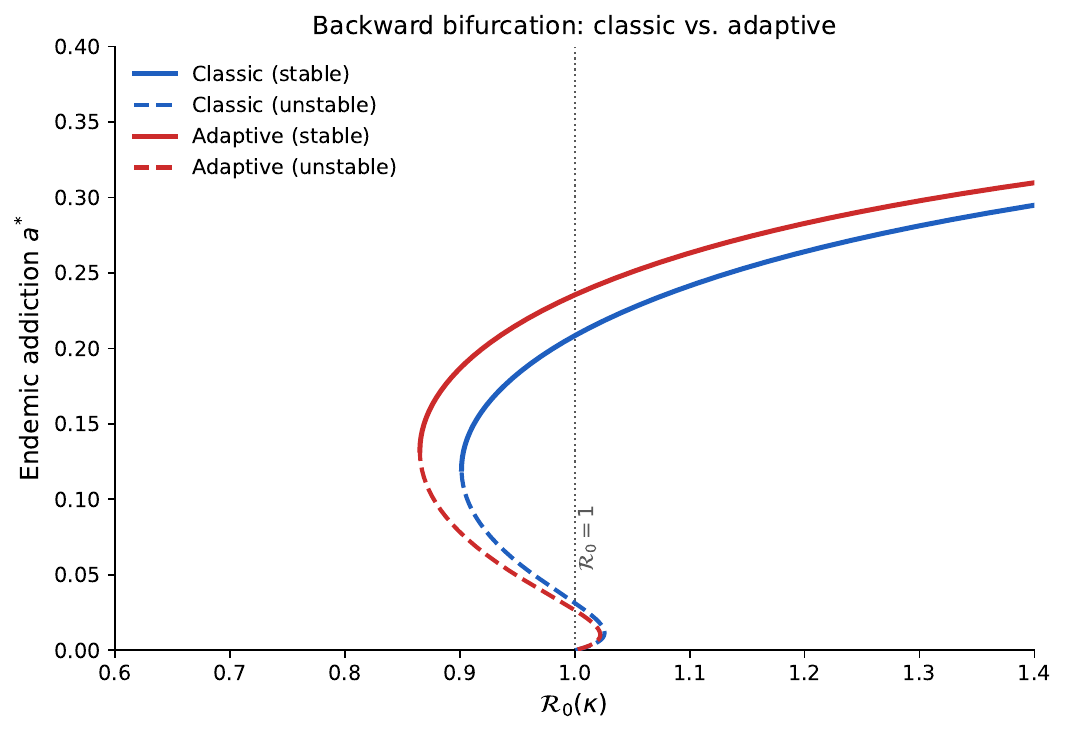}
\caption{Backward bifurcation for the classical (blue) and adaptive (red) systems. Solid: stable branches; dashed: unstable. The adaptive endemic branch sits higher and its fold shifts to lower $\Rz$ ($\Delta\Rz\approx-0.035$).}
\label{fig:bifurcation}
\end{figure}

\subsection{Enlargement of the endemic basin}
\label{subsec:basin}
In the bistable window, the outcome depends on the initial fraction of the population addicted, $a(0)$. For each $\Rz$ we locate, by bisection, the critical seed $a(0)_{\mathrm{crit}}$ above which trajectories converge to the endemic state rather than to extinction. Figure~\ref{fig:basin} and Table~\ref{tab:basin} show that the adaptive threshold lies \emph{below} the classical one across the entire window: the shaded gap is the set of initial conditions that go extinct under the classical model but persist under adaptation. The threshold reduction increases as $\Rz$ approaches the classical fold, from $2.1$ percentage points at $\Rz=0.98$ to $5.9$ percentage points at $\Rz=0.92$. The effect can reverse the qualitative outcome: at $\Rz=0.92$ with $a(0)=0.18$ between the classical threshold ($0.211$) and the adaptive one ($0.152$), the classical model converges to $a^*=0$ while the adaptive model sustains $a^*\approx0.20$. The contrast is sharpest just below the classical fold: at $\Rz=0.90$ the classical model admits \emph{no} endemic state at all (its fold lies at $\Rc^{\mathrm{cl}}\approx0.901$), yet the adaptive model sustains addiction from any seed $a(0)\gtrsim0.19$, because its fold has moved to $\Rcb\approx0.866$. These thresholds are robust to the integration horizon, holding to $\pm0.005$ for $t$ up to $2\times10^5$ days in the well-defined range $\Rz\ge0.92$.

\begin{table}[h!]
\centering
\caption{Critical initial addiction $a(0)_{\mathrm{crit}}$ for the endemic outcome. A lower threshold means a larger endemic basin. Below the classical fold ($\Rz<\Rc^{\mathrm{cl}}\approx0.901$) the classical model has no endemic state, while the adaptive model retains one down to $\Rcb\approx0.866$.}
\label{tab:basin}
\small
\begin{tabular}{cccc}
\toprule
$\Rz$ & classical & adaptive & reduction (pp) \\
\midrule
$0.90$ & none & $0.190$ & --- \\
$0.92$ & $0.211$ & $0.152$ & $5.9$ \\
$0.94$ & $0.161$ & $0.120$ & $4.0$ \\
$0.96$ & $0.121$ & $0.092$ & $2.9$ \\
$0.98$ & $0.087$ & $0.067$ & $2.1$ \\
\bottomrule
\end{tabular}
\end{table}

\begin{figure}[ht!]
\centering
\includegraphics[width=0.72\textwidth]{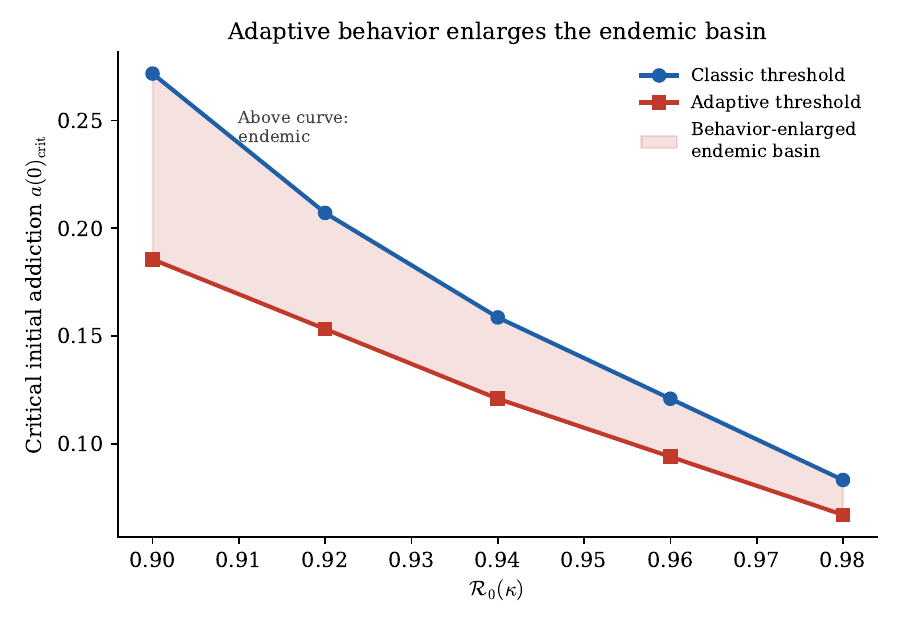}
\caption{Critical initial addiction $a(0)_{\mathrm{crit}}$ versus $\Rz$. The adaptive threshold (red) lies below the classical (blue); the shaded region is the behavior-enlarged endemic basin initial states that go extinct under the classical model but persist under adaptation.}
\label{fig:basin}
\end{figure}

\subsection{Sensitivity of the effect size to behavioral parameters}
\label{subsec:sensitivity}
The \emph{direction} of the effect is fixed by the sign of $M-1$ and is therefore robust by Theorem~\ref{thm:general}; its \emph{magnitude}, however, depends on the behavioral parameters, and Table~\ref{tab:sensitivity} reports how the adaptive fold $\Rcb$ and the critical-seed reduction $\Delta a_{\mathrm{crit}}$ (at $\Rz=0.95$) respond to the two that matter most: the activity gap $\rho=b^a/b^s$ and the planning horizon $\tau$. The activity gap is the dominant driver, exactly as Eq.~\eqref{eq:Mstatic} predicts: as $\rho\to1$ the addicted become as socially active as susceptibles, the
mixing premium $M_{\text{static}}=1/[1-a(1-\rho)]\to1$, and the effect vanishes continuously the fold shift contracts from $-0.057$ at $\rho=0.5$ to $-0.009$ at $\rho=0.9$, with $\Delta a_{\mathrm{crit}}$ falling in step from $4.6$ to $1.2$
percentage points. The planning horizon, by contrast, has only a weak influence over the addiction-relevant range ($\tau=6$ to $24$ days shifts the fold by less than
$0.003$, and leaves $\Delta a_{\mathrm{crit}}$ unchanged at $3.4$ percentage points); a longer horizon induces a slightly larger precautionary reduction, which by $\partial M/\partial\Cs>0$ nudges the fold back toward the classical value, the same attenuation mechanism discussed in Section~\ref{sec:discussion}. The effect size is thus
governed by who the addicted are (how much they withdraw from ordinary social activity), not by how far ahead at-risk individuals plan.

\begin{table}[h!]
\centering
\caption{Sensitivity of the effect size to the activity gap $\rho=b^a/b^s$ and the planning horizon $\tau$. The classical fold ($0.9014$) and classical critical seed at $\Rz=0.95$ ($a_{\mathrm{crit}}=0.140$) are independent of both. $\Delta a_{\mathrm{crit}}$ is the reduction in the critical seed at $\Rz=0.95$ (percentage points); the $\tau$ block holds $\rho=0.667$ and the $\rho$ block holds $\tau=12$. Baseline row in bold.}
\label{tab:sensitivity}
\small
\begin{tabular}{cccc}
\toprule
varied & adaptive fold $\Rcb$ & fold shift $\Delta\Rz$ & $\Delta a_{\mathrm{crit}}$ (pp) \\
\midrule
Classical ($M\equiv1$, reference) & $0.9014$ & $0$ & $0$ \\
\midrule
\multicolumn{4}{l}{\emph{Activity gap} $\rho=b^a/b^s$ (at $\tau=12$):}\\
$\rho=0.50$            & $0.845$ & $-0.057$ & $+4.6$ \\
$\boldsymbol{\rho=0.667}$ & $\mathbf{0.865}$ & $\mathbf{-0.036}$ & $\mathbf{+3.4}$ \\
$\rho=0.80$            & $0.881$ & $-0.021$ & $+2.2$ \\
$\rho=0.90$            & $0.892$ & $-0.009$ & $+1.2$ \\
\midrule
\multicolumn{4}{l}{\emph{Planning horizon} $\tau$ (days, at $\rho=0.667$):}\\
$\tau=6$               & $0.865$ & $-0.037$ & $+3.4$ \\
$\boldsymbol{\tau=12}$ & $\mathbf{0.865}$ & $\mathbf{-0.037}$ & $\mathbf{+3.4}$ \\
$\tau=24$              & $0.867$ & $-0.034$ & $+3.4$ \\
\bottomrule
\end{tabular}
\end{table}

\subsection{Stochastic establishment in finite populations}
\label{subsec:stochastic}
The deterministic basin enlargement predicts that, in a finite population, behavioral adaptation should raise the \emph{probability} that an addicted seed establishes rather than fading out. We test this on the continuous-time Markov chain of \citep[Table~1]{sanchez2023}, the stochastic analog of \eqref{eq:sar_adaptive}, modified only so that the addiction (transmission) rate carries the mixing factor $M$ of Lemma~\ref{lem:Mfactor} ($M\equiv1$ recovers the classical chain). We simulate via the Gillespie algorithm \citep{Gillespie1976,Gillespie1977}, seeding $A(0)=\lceil a_0 N_0\rceil$ addicted individuals and recording the fraction of $1000$ realizations that reach the upper endemic branch ($a>0.10$ at $t=3\times10^4$ days) rather than going extinct.

Figure~\ref{fig:stochastic} shows establishment-probability curves at $\Rz=0.95$, $N_0=1500$, where the deterministic critical seeds are $a_{\mathrm{crit}}^{\text{adap}}\approx0.106$ and $a_{\mathrm{crit}}^{\text{class}}\approx0.140$. Both curves rise with the seed fraction, but the adaptive curve lies systematically above the classical one. At $a_0=0.10$ just below \emph{both} deterministic thresholds, where the mean-field models predict extinction for both addiction, nonetheless establishes stochastically in $29\%$ of adaptive runs versus $11\%$ of classical runs, a roughly threefold difference attributable solely to behavioral adaptation. The gap is widest in and just below the threshold window because the adaptive branch reaches lower $\Rz$ and smaller seeds; finite-size fluctuations are more likely to carry an adaptive population over to the endemic attractor. The deterministic basin enlargement is thus not an artifact of the mean-field limit; it manifests as a genuine increase in the stochastic establishment probability, the quantity most relevant to whether a nascent addiction cluster takes hold in a real community of finite size.

\begin{figure}[ht!]
\centering
\includegraphics[width=0.72\textwidth]{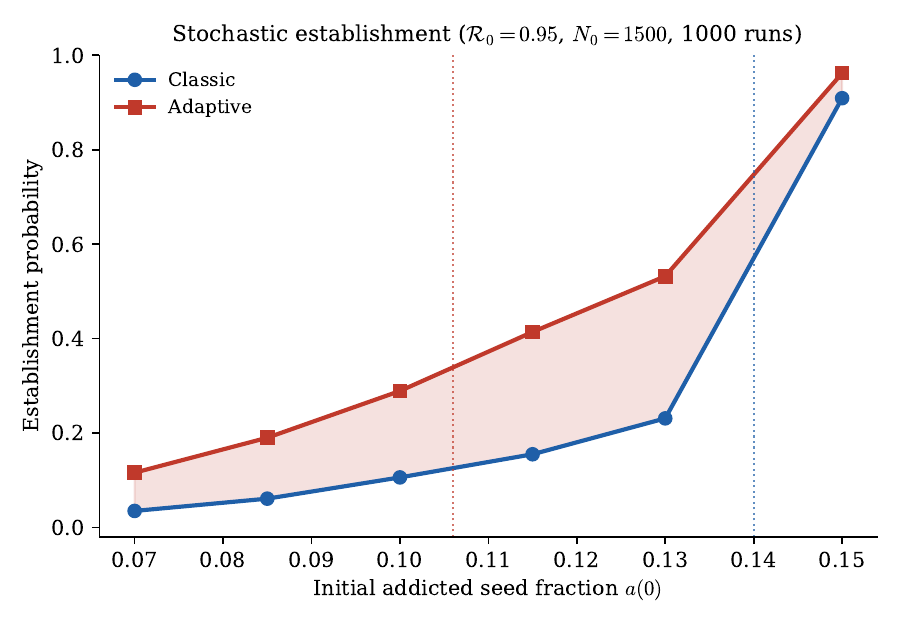}
\caption{Stochastic establishment probability versus initial addicted seed fraction $a(0)$, from Gillespie simulation of the CTMC ($\Rz=0.95$, $N_0=1500$, $1000$ realizations). Dotted lines mark the deterministic critical thresholds (adaptive $a_{\mathrm{crit}}\approx0.106$, classical $a_{\mathrm{crit}}\approx0.140$); between them, the adaptive chain (red) establishes addiction far more often than the classical chain (blue).}
\label{fig:stochastic}
\end{figure}

\paragraph{Robustness across population size and $\Rz$.}
Figure~\ref{fig:stochrobust} confirms that the establishment gap is not specific to one operating point. Panel (a) varies the population size $N_0$ at $\Rz=0.95$, $a_0=0.10$: the adaptive establishment probability exceeds the classical one at every $N_0$ from $500$ to $4000$ (for example, $0.28$ versus $0.09$ at $N_0=1000$), with $95\%$ Wilson confidence intervals that do not overlap. Both probabilities decline with $N_0$; larger populations track the deterministic prediction of extinction below the classical threshold, but the adaptive curve declines more slowly, consistent with its lower deterministic threshold. Panel (b) varies $\Rz$ at $N_0=1500$: the gap is widest at and just below $\Rz=1$ ($0.74$ versus $0.49$ at $\Rz=1.00$, non-overlapping intervals) and closes at both extremes, as both chains approach certain establishment for large $\Rz$ and certain extinction for small $\Rz$ the same near-threshold concentration seen deterministically in Figure~\ref{fig:effectsize}. All probabilities are Monte Carlo estimates based on $1000$ realizations, with Wilson $95\%$ score intervals \citep{Wilson1927,BrownCaiDasGupta2001}.

\begin{figure}[ht!]
\centering
\includegraphics[width=0.95\textwidth]{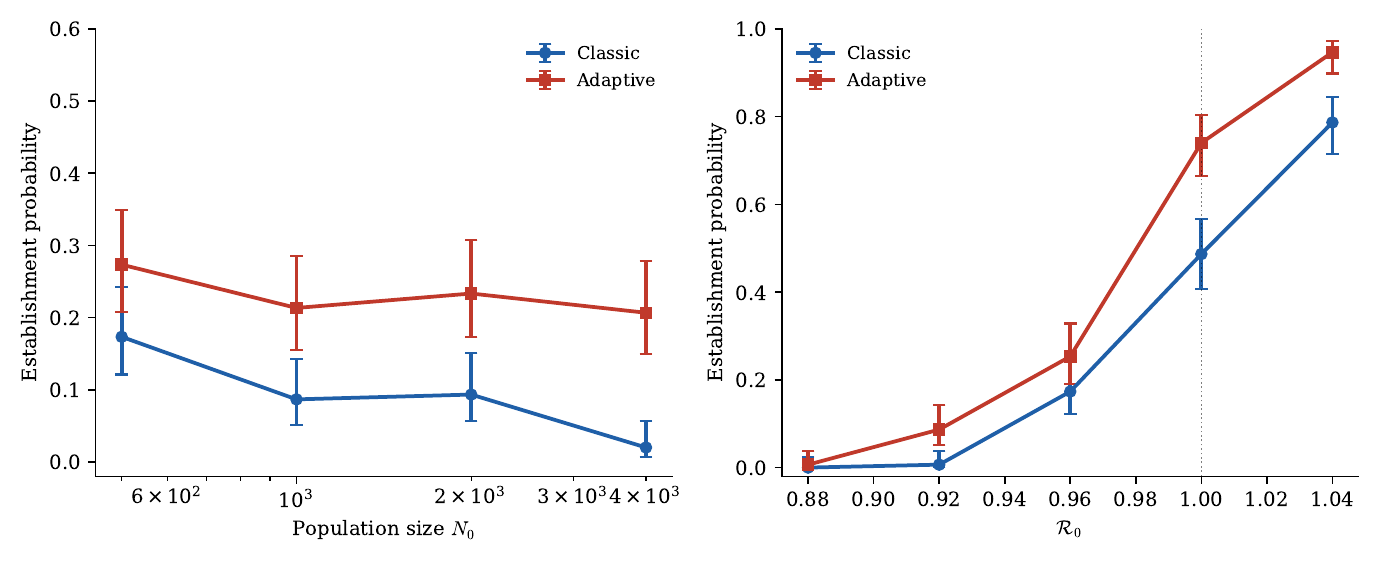}
\caption{Stochastic establishment probability with Wilson $95\%$ confidence intervals. (a) Versus population size $N_0$ at $\Rz=0.95$; (b) versus $\Rz$ at $N_0=1500$. The adaptive chain (red) establishes addiction more often than the classical chain (blue) throughout, with the largest gap near $\Rz=1$.}
\label{fig:stochrobust}
\end{figure}

\section{Why Self-Protection Sustains Addiction}
\label{sec:mechanism}

Section~\ref{sec:general} reduces the entire question to the sign of $M-1$: the fold moves left, enlarging the endemic basin, precisely because $M\ge1$ along the branch (Lemma~\ref{lem:m1pos}). But why should rational self-protection produce $M\ge1$ an \emph{increase} in effective transmission rather than the decrease one would naively expect? This section gives the mechanistic answer. Two distinct effects, both consequences of the conditional proportional mixing \eqref{eq:foi}, combine to keep $M$ above unity: the first concerns how transmission responds to a susceptible's \emph{own} contact choice, the second the \emph{baseline} level of transmission even absent any behavioral change.

\paragraph{(i) Self-protection is sub-linearly effective.}
Fix the population at a representative endemic equilibrium of the adaptive system (at $\Rz=0.95$, $(s,a,\st)=(0.256,0.215,0.530)$) and ask how the force of addiction responds as a susceptible cuts their own contact $\Cs$ below the static optimum (Figure~\ref{fig:mechanism}). Because $\Cs$ enters \emph{both} the numerator and the mixing denominator $\bar C=s\Cs+a\Ca+\st\Cr$ of \eqref{eq:foi}, cutting contact also shrinks the shared contact pool, raising the conditional probability that each remaining contact is with an addicted individual. The two effects partly offset, so $\lambda$ falls sub-linearly: cutting $\Cs$ from $5.0$ to
$3.0$, a $40\%$ reduction, lowers $\lambda$ by only $33\%$. A susceptible cannot withdraw from addicted contacts as fast as from contacts in general, so self-protection buys less risk reduction than its effort would suggest, and the optimizing agent, anticipating this, cuts contact only modestly to begin with.

\paragraph{(ii) Proportional mixing exceeds frequency dependence off-equilibrium.} The second effect is present \emph{even if the susceptible does not adapt at all}. Evaluating \eqref{eq:foi} at the static optimum $\Cs=C^{s*}$, the conditional proportional form still delivers $7.7\%$ higher transmission than the frequency-dependent term $\beta g(\st)a$ of the classical model at this endemic state (Figure~\ref{fig:mechanism}, the gap at $\Cs=5$). The reason is structural: the calibration \eqref{eq:calib} forces the two forms to coincide at the addiction-free equilibrium, but away from it the mean contact level $\bar C$ falls
below $C^{s*}$ as the population shifts into the lower-activity addicted and reformed classes ($\Ca<C^{s*}$). A smaller $\bar C$ inflates the per-contact probability of meeting an addicted individual, so $M=C^{s*}/\bar C>1$. This is the mixing effect of Eq.~\eqref{eq:Mstatic}, here seen as a transmission premium that the classical model omits.

Mechanisms (i) and (ii) are distinct in a comparative static in the agent's own choice, the other a level difference at the optimum, but they push the same way, and together they keep $M\ge1$ throughout the branch (Lemma~\ref{lem:m1pos}). By the comparison principle (Theorem~\ref{thm:general}), the effective transmission at the endemic equilibrium therefore exceeds the classical value, raising the upper branch and pulling the fold to lower $\Rz$. Rational self-protection backfires not because agents misjudge their risk, but because the geometry of proportional mixing converts a population's withdrawal into a concentration of residual risk.

\begin{figure}[ht!]
\centering
\includegraphics[width=0.72\textwidth]{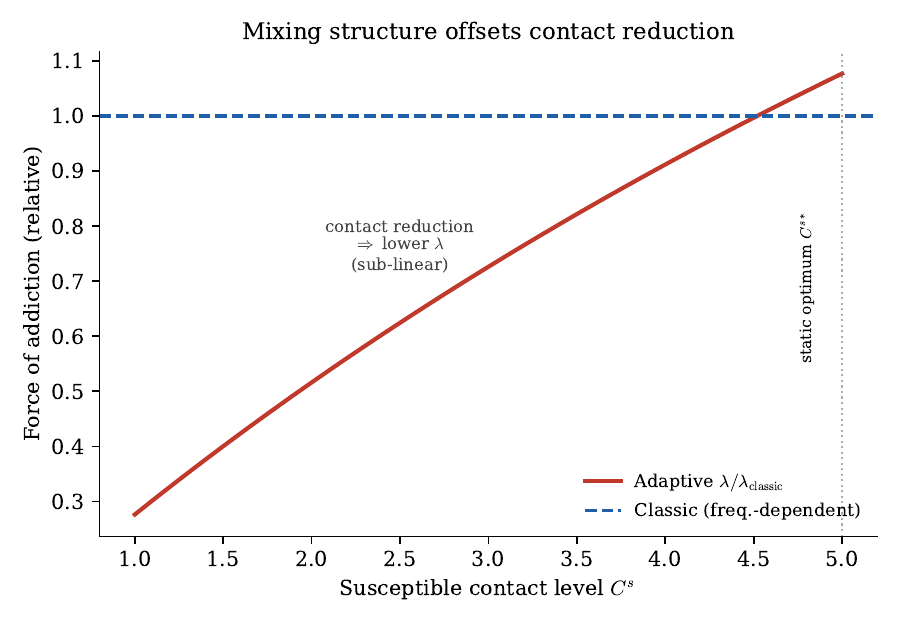}
\caption{Force of addiction relative to the classical frequency-dependent value, as a function of susceptible contact $\Cs$, at a representative adaptive endemic equilibrium ($\Rz=0.95$, $(s,a,\st)=(0.256,0.215,0.530)$). Reducing $\Cs$ lowers $\lambda$ only sub-linearly (mechanism i); already at the static optimum $\Cs=C^{s*}$ the proportional mixing form exceeds frequency dependence by $7.7\%$ (mechanism ii).}
\label{fig:mechanism}
\end{figure}

\subsection{Contingency on the mixing form: a model-discriminating prediction}
\label{subsec:mixingform}
Corollary~\ref{cor:dichotomy} makes the dependence on the mixing assumption sharp rather than hidden. Because the behavioral layer enters only through $M$, the sign of the basin effect is fixed by which side of $M\equiv1$ the chosen coupling falls on, and the two standard couplings fall on opposite sides:
\begin{itemize}
\item \textbf{Conditional proportional mixing} (this paper):
  $M=C^{s}/(s\Cs+a\Ca+\st\Cr)\ge1$ along the branch (Lemma~\ref{lem:m1pos}). The denominator normalization is exactly what makes a susceptible's \emph{relative} exposure rise as the population deactivates, so $M>1$ and the basin \emph{enlarges}.
\item \textbf{Frequency-dependent coupling}: if the contact decision instead scales transmission linearly, without renormalizing by the mean contact level, $\lambda=\beta g(\st)\,a\,(\Cs/c_0)$, then $M=\Cs/c_0\le1$. By Theorem~\ref{thm:general}, the fold moves \emph{right} and the basin \emph{shrinks}.
\end{itemize}
The two assumptions thus make \emph{opposite} qualitative predictions about the sign of the behavioral effect (Figure~\ref{fig:mixingform}). This is more than a caveat: it identifies the contact-scaling law as the sole modeling choice that determines the direction of the effect and is, in principle, model-discriminating. A study able to measure both the change in individual contact during an addiction wave and the resulting change in incidence could determine which coupling governs the data, and hence whether voluntary precaution accelerates or retards establishment. We adopt the conditional proportional form as the more defensible default it is standard in the epi-economic literature \citep{fenichel2011}, and encodes the realistic feature that withdrawing from general social activity does not proportionally withdraw one from contact with \emph{addicted} individuals but, rather than leave this assumption implicit, the comparison principle makes the alternative's prediction explicit and, at least in principle, falsifiable.

\paragraph{A testable implication, and an open measurement question.}
The dichotomy yields a concrete empirical signature. Suppose one tracks, over the course of an addiction wave, both the typical contact level of at-risk individuals and the per-susceptible incidence of new addiction. Under frequency-dependent coupling, the two move together: a measured drop in individual contact translates into a proportional drop in incidence, so voluntary precaution visibly retards spread. Under conditional proportional mixing, it does not because withdrawal concentrates the residual contacts of the still-active among the addicted; a measured drop in average contact is accompanied by incidence that falls less than proportionally, holds steady, or even rises, the empirical fingerprint of $M>1$. The sign of the contact--incidence co-movement, conditional on prevalence, therefore discriminates the two kernels directly, without recourse to the model. Realizing this test requires longitudinal, individual-level contact data in a substance-use setting who the at-risk associate with, and how that composition shifts as prevalence climbs paired with incidence; the social-contact surveys that have made this kind of measurement routine for respiratory pathogens \citep{Mossong2008} have, to our knowledge, no established counterpart for addiction, where the relevant ``contact'' is diffuse social exposure rather than a discrete encounter and is correspondingly harder to define and observe. We therefore present the contact-scaling law not as a settled modeling choice but as an open measurement question whose resolution our framework reduces to a single observable sign.

\begin{figure}[ht!]
\centering
\includegraphics[width=0.70\textwidth]{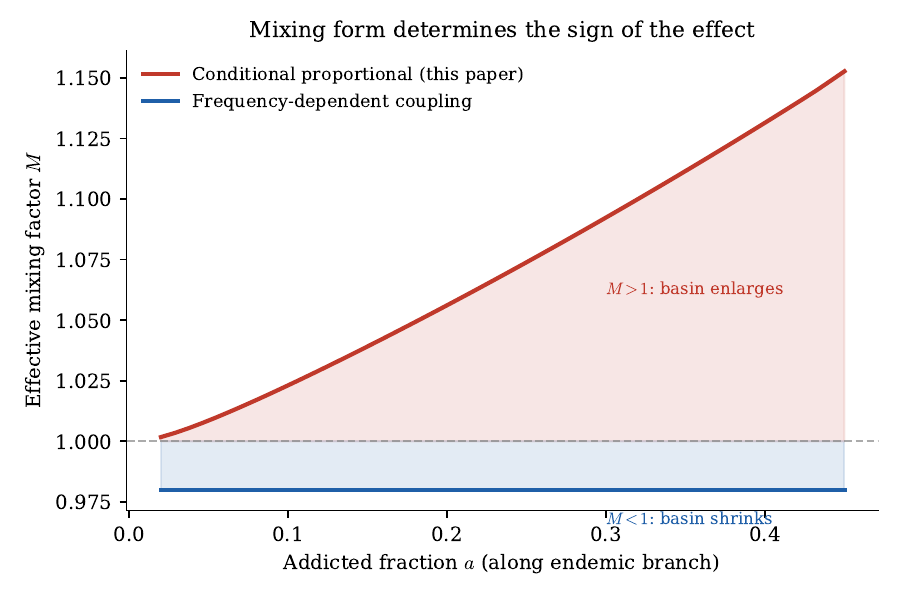}
\caption{Effective mixing factor $M$ along the endemic branch under conditional proportional mixing (red, $M>1$, basin enlarges) versus frequency-dependent coupling (blue, $M<1$, basin shrinks). The two assumptions make opposite qualitative predictions.}
\label{fig:mixingform}
\end{figure}

\section{Policy Implications}
\label{sec:policy}

The model is deliberately stylized and calibrated with behavioral parameters borrowed from the infectious-disease setting (Section~\ref{sec:numerical}), so the implications below are qualitative and mechanism-level rather than operational. Each follows from a structural feature of the analysis and would need calibration to addiction data before informing a specific intervention.

\paragraph{Eradication thresholds from behavior-free models are optimistic.}
The classical SAR model implies that pushing $\Rz$ below the fold $\Rc^{\mathrm{cl}}$ guarantees that addiction dies out. The comparison principle (Theorem~\ref{thm:general}) and Figure~\ref{fig:basin} show that, under proportional mixing, adaptation lowers the effective fold to $\Rcb<\Rc^{\mathrm{cl}}$: a residual bistable region survives between the two folds, within which an established epidemic persists even though the behavior-free model predicts elimination. A control program calibrated to the classical threshold, therefore, aims at a target that is too weak. This is the most robust policy reading of our results, as it follows directly from the sign of $M-1$ and does not depend on the magnitude of the behavioral response. The difficulty of eradicating an entrenched epidemic despite sustained control salient in the trajectory of the US opioid crisis, where the burden has continued to climb under intensive intervention \citep{Zhang2024} is exactly the hysteresis that a backward bifurcation formalizes, and that behavioral adaptation here deepens.

\paragraph{Public deterrence and private caution are substitutes, weakly.}
The willingness factor $\nu$ (public deterrence by the reformed population) and the private contact reduction $C^{s*}-\Csa$ act in the same direction in the first-order condition \eqref{eq:foc}: stronger public deterrence lowers $g(\st)$, hence the marginal addiction risk, hence the private incentive to cut contact, so $\Csa$ rises with $\nu$. This crowding-out is the addiction analog of the public-private substitution noted by \citep{fenichel2011}. In our model, the effect is present but small in the addiction-relevant regime; raising $\nu$ from $0$ to $1$ increases $\Csa$ by under $1\%$ at baseline and grows only as transmission intensity and the planning horizon increase (reaching the order of tens of percent for influenza-scale $\beta$ and multi-week horizons). The qualitative caution stands that public and private mitigation are not simply additive, but the model does not support treating the substitution as quantitatively important at addiction timescales.

\paragraph{Encouraging self-protection has an ambiguous net effect.}
The behavioral effect on the endemic level is concentrated near $\Rz=1$
(Figure~\ref{fig:effectsize}), so interventions that elicit precautionary contact reduction have the greatest leverage in populations near the bifurcation. The direction of that leverage, however, is not guaranteed to be benign. Under conditional proportional mixing, the analysis of Section~\ref{sec:mechanism} shows that individual contact reduction is partially self-defeating: because withdrawal concentrates a susceptible's residual contacts among the addicted, encouraging self-protection can \emph{enlarge} the endemic basin rather than shrink it. Whether it helps or harms depends on the contact-scaling law (Corollary~\ref{cor:dichotomy}), which the present model cannot resolve from first principles. The policy implication is therefore a caution rather than a prescription: campaigns that promote individual avoidance during an addiction wave should not be assumed to reduce long-run prevalence, and their sign should be established empirically by measuring how contact and incidence co-move before they are relied upon.

\section{Discussion}
\label{sec:discussion}

Embedding adaptive contact behavior in the SAR addiction model leaves the basic reproductive number and the addiction-free stability condition untouched (Theorem~\ref{thm:R0}) but reshapes the endemic landscape through a single, exactly characterized channel. The behavioral layer collapses to a scalar modifier $M$ (Lemma~\ref{lem:Mfactor}); along the endemic branch the bifurcation curve factorizes as $\Rz(a)=\Rcl(a)/M(a)$ (Lemma~\ref{lem:branch}); and the saddle-node fold moves left iff $M\ge1$ and right iff $M\le1$ (Theorem~\ref{thm:general}), a dichotomy independent of the mixing form. Self-protective behavior under proportional mixing satisfies $M\ge1$ (Lemma~\ref{lem:m1pos}), so it enlarges the endemic basin; the effect is concentrated near $\Rz=1$ (Figure~\ref{fig:effectsize}), shifts the fold by $\Delta\Rz\approx-0.035$ (Figure~\ref{fig:bifurcation}), lowers the critical initial addiction (Figure~\ref{fig:basin}, Table~\ref{tab:basin}), and raises stochastic establishment probability up to threefold near threshold (Figures~\ref{fig:stochastic}--\ref{fig:stochrobust}).

Three caveats bound the result, and we have made each precise rather than rhetorical. (i) The comparison principle (Theorem~\ref{thm:general}) is exact and mixing-form-agnostic, but applying it to the adaptive model requires $M\ge1$, which holds under the sufficient condition of Lemma~\ref{lem:m1pos} satisfied throughout the addiction-relevant regime (all $\tau\le82$ days, all $\beta\le0.052$) but capable of failing for implausibly long horizons or influenza-scale transmission, where strong behavioral withdrawal would push $M$ below $1$ and, by the same theorem, reverse the direction. (ii) The branch modifier $M(a)$ depends on the optimal contact $\Csa$, which itself depends on $\Rz$ through the FOC; the factorization \eqref{eq:Rcurve} is therefore a fixed point evaluated at $\Rz=\Rcb$, and the $M\ge1$ bound is verified to be self-consistent at the fold. (iii) The behavioral model assumes homogeneous, rational agents with adaptive expectations and an exogenous planning horizon. The behavioral response is modest in absolute terms at the slow timescales of addiction ($\beta\kappa\approx0.003\,\mathrm{day}^{-1}$, some $30\times$ slower than influenza in \citealp{fenichel2011}); the qualitative basin effect survives nonetheless deterministically (Figure~\ref{fig:basin}) and stochastically across $N_0$ and $\Rz$ (Figure~\ref{fig:stochrobust}) because it operates precisely where the system is most sensitive.

Our use of adaptive rather than fully rational expectations invites an obvious question: would a forward-looking agent who anticipates the entire prevalence trajectory as in a Hamilton--Jacobi--Bellman (HJB) or differential-game formulation \citep{reluga2010,Toxvaerd2020} overturn the result? The comparison principle (Theorem~\ref{thm:general}) makes the answer unusually clean, because it is indifferent to \emph{how} the contact level is formed: the basin enlarges if and only if $M\ge1$ along the branch, whatever the expectations scheme, so only the magnitude and possibly the sign of $M-1$ is at issue. Here the direction runs opposite to naive intuition. An HJB agent anticipating rising prevalence cuts contact earlier and more sharply than an adaptive-expectations agent; since $M$ is increasing in the susceptible's own contact ($\partial M/\partial\Cs>0$), this \emph{larger} precautionary reduction pushes $M$ toward unity and thus \emph{attenuates} the basin enlargement rather than amplifying it. The qualitative result is robust nonetheless, because the inequality $M\ge1$ rests on the structural mixing premium $M_{\text{static}}\ge1$ of Eq.~\eqref{eq:Mstatic} present even absent any behavioral change and independent of the expectations scheme against which the precautionary reduction works but which it need not overturn. We therefore expect a fully rational treatment to preserve the qualitative result for the modest reductions characteristic of addiction's slow timescale while shrinking its magnitude; only under implausibly aggressive precaution would $M$ fall below one and, by the same theorem, reverse the sign precisely the single inequality such a treatment would need to re-check.

Natural next steps are (a) a social planner formulation to quantify the crowding-out wedge; (b) heterogeneous susceptibility $b^a$, yielding a distribution of contact choices; and (c) the fully rational (Hamilton--Jacobi--Bellman) treatment discussed above, which would best-respond to the anticipated trajectory rather than to current prevalence. The comparison principle extends verbatim to any of these as long as the resulting modifier $M$ admits the branch factorization \eqref{eq:Rcurve}; only the sign of $M-1$ needs to be re-checked. The stochastic establishment analysis begun in Section~\ref{subsec:stochastic} could be extended to quasi-stationary distributions and mean times to extinction.

\appendix

\section{Derivation of \texorpdfstring{$V_k(a)$}{Vk(a)}}
\label{app:Vi}
An addicted individual recovers with per-period probability $P^r=1-e^{-\gamma}$, upon which they enjoy the reformed flow utility $u^r=u^r(C^{r*})$ in each subsequent period; while addicted, the flow utility is $u^a(C^{a*})=0$ by the scaling of Section~\ref{sec:behavior}. With $k$ periods remaining, the value of being addicted satisfies
\[
V_k(a)=\delta\bigl[(1-P^r)V_{k-1}(a)+P^r V_{k-1}(r)\bigr],\qquad
V_{k-1}(r)=u^r\,\frac{1-\delta^{k-1}}{1-\delta},\quad V_0(a)=0,
\]
where $V_{k-1}(r)$ is the discounted reformed stream evaluated as of the period of recovery, and the leading $\delta$ discounts the one-period wait until the transition is realized. The homogeneous part contracts at rate $\delta(1-P^r)$; solving the linear recurrence gives the closed form \eqref{eq:Vi}, which we have verified both symbolically and by Monte Carlo simulation of the underlying absorbing recovery process.

\section{Numerical methods}
\label{app:numerics}
The Bellman problem \eqref{eq:bellman} is solved by backward induction on a
$100$--$200$ point grid for $\Cs\in[0,b^s/2]$, with $V_k(a)$ from \eqref{eq:Vi} and a rolling horizon $\tau$. The ODE \eqref{eq:sar_adaptive} is integrated with LSODA \citep{Petzold1983} (\texttt{rtol}=$10^{-8}$, \texttt{atol}=$10^{-10}$). Endemic equilibria are obtained by long-time integration ($t$ up to $2\times10^5$ days) from high and low initial conditions; critical thresholds in Table~\ref{tab:basin} by bisection to $\pm0.005$. The branch-minimum fold \eqref{eq:foldmin} is evaluated on a $10^5$-point grid in $a$; it reproduces the closed-form classical fold (\citealp[Eq.~5]{sanchez2023}, $\nu=0$) to machine precision and the directly continued adaptive fold ($\nu=0.8$) to three digits. The unstable middle branch
in Figure~\ref{fig:bifurcation} is obtained by root-finding on the steady-state condition $\dot a=0$ along the slow manifold \eqref{eq:branchstate}.

Single-peakedness of the susceptible's objective \eqref{eq:bellman} was verified numerically by evaluating the objective on the same $200$-point grid in $\Cs\in[0,b^s/2]$ used by the solver, at $1{,}920$ states formed as the Cartesian product of a grid over the planning horizon $\tau$, the utility curvature $\gamma_u$, and the basic reproductive number $\Rz$ across the ranges of Section~\ref{sec:behavior} ($\tau\in[6,60]$ days, $\gamma_u\in[0.1,1]$, $\Rz\in[0.90,1.00]$) with the endemic-branch state \eqref{eq:branchstate} at each $\Rz$; the objective exhibited a unique interior maximum in every case, with no instance of a second local maximum or a boundary optimum.


\begin{thebibliography}{10}

\bibitem{bauch2013}
C.~T. Bauch and A.~P. Galvani.
\newblock Social factors in epidemiology.
\newblock {\em Science}, 342(6154):47--49, 2013.

\bibitem{Behrens2000}
D.~A. Behrens, J.~P. Caulkins, G.~Tragler, and G.~Feichtinger.
\newblock Optimal control of drug epidemics: prevent and treat but not at the same time?
\newblock {\em Manage. Sci.}, 46(3):333--347, 2000.

\bibitem{BlytheCastilloChavez1989}
S.~P. Blythe and C.~Castillo-Chavez.
\newblock Mixing framework for social/sexual behavior.
\newblock In {\em Mathematical and Statistical Approaches to AIDS Epidemiology}, volume~83 of {\em Lecture Notes in Biomathematics}, pages 275--288. Springer, 1989.

\bibitem{Brauer2004}
F.~Brauer.
\newblock Backward bifurcations in simple vaccination models.
\newblock {\em J. Math. Anal. Appl.}, 298(2):418--431, 2004.

\bibitem{BrownCaiDasGupta2001}
L.~D. Brown, T.~T. Cai, and A.~DasGupta.
\newblock Interval estimation for a binomial proportion.
\newblock {\em Stat. Sci.}, 16(2):101--117, 2001.

\bibitem{BusenbergCastilloChavez1991}
S.~Busenberg and C.~Castillo-Chavez.
\newblock A general solution of the problem of mixing of subpopulations and its application to risk- and age-structured epidemic models for the spread of aids.
\newblock {\em IMA J. Math. Appl. Med. Biol.}, 8(1):1--29, 1991.

\bibitem{castillochavez2004}
C.~Castillo-Chavez and B.~Song.
\newblock Dynamical models of tuberculosis and their applications.
\newblock {\em Math. Biosci. Eng.}, 1(2):361--404, 2004.

\bibitem{Dong2025}
D.~Chen, Y.~Sun, X.~Li, and Z.~Yang.
\newblock Global burden of drug use disorders from 1990 to 2021 and projections to 2046.
\newblock {\em Front. Public Health}, 13:1550518, 2025.

\bibitem{Chen2009}
F.~H. Chen.
\newblock Modeling the effect of information quality on risk behavior change and the transmission of infectious diseases.
\newblock {\em Math. Biosci.}, 217(2):125--133, 2009.

\bibitem{DushoffHuangCastilloChavez1998}
J.~Dushoff, W.~Huang, and C.~Castillo-Chavez.
\newblock Backwards bifurcations and catastrophe in simple models of fatal diseases.
\newblock {\em J. Math. Biol.}, 36(3):227--248, 1998.

\bibitem{fenichel2011}
E.~P. Fenichel et~al.
\newblock Adaptive human behavior in epidemiological models.
\newblock {\em Proc. Natl. Acad. Sci. USA}, 108(15):6306--6311, 2011.

\bibitem{funk2010}
S.~Funk, M.~Salathé, and V.~A.~A. Jansen.
\newblock Modelling the influence of human behaviour on the spread of infectious diseases: a review.
\newblock {\em J. R. Soc. Interface}, 7(50):1247--1256, 2010.

\bibitem{GeoffardPhilipson1996}
P.-Y. Geoffard and T.~Philipson.
\newblock Rational epidemics and their public control.
\newblock {\em Int. Econ. Rev.}, 37(3):603--624, 1996.

\bibitem{Gillespie1976}
D.~T. Gillespie.
\newblock A general method for numerically simulating the stochastic time evolution of coupled chemical reactions.
\newblock {\em J. Comput. Phys.}, 22(4):403--434, 1976.

\bibitem{Gillespie1977}
D.~T. Gillespie.
\newblock Exact stochastic simulation of coupled chemical reactions.
\newblock {\em J. Phys. Chem.}, 81(25):2340--2361, 1977.

\bibitem{Gumel2012}
A.~B. Gumel.
\newblock Causes of backward bifurcations in some epidemiological models.
\newblock {\em J. Math. Anal. Appl.}, 395(1):355--365, 2012.

\bibitem{hadeler1995}
K.~P. Hadeler and C.~Castillo-Chavez.
\newblock A core group model for disease transmission.
\newblock {\em Math. Biosci.}, 128:41--55, 1995.

\bibitem{hethcote2000}
H.~W. Hethcote.
\newblock The mathematics of infectious diseases.
\newblock {\em SIAM Rev.}, 42(4):599--653, 2000.

\bibitem{HethcoteYorke1984}
H.~W. Hethcote and J.~A. Yorke.
\newblock {\em Gonorrhea Transmission Dynamics and Control}, volume~56 of {\em Lecture Notes in Biomathematics}.
\newblock Springer, 1984.

\bibitem{Jacquez1988}
J.~A. Jacquez, C.~P. Simon, J.~Koopman, L.~Sattenspiel, and T.~Perry.
\newblock Modeling and analyzing hiv transmission: the effect of contact patterns.
\newblock {\em Math. Biosci.}, 92(2):119--199, 1988.

\bibitem{Memarbashi2022}
R.~Memarbashi, A.~Ghasemabadi, and Z.~Ebadi.
\newblock Backward bifurcation in a two strain model of heroin addiction.
\newblock {\em Comput. Methods Differ. Equ.}, 10(3):656--673, 2022.

\bibitem{Mossong2008}
J.~Mossong et~al.
\newblock Social contacts and mixing patterns relevant to the spread of infectious diseases.
\newblock {\em PLoS Med.}, 5(3):e74, 2008.

\bibitem{UNODC2025}
United Nations~Office on~Drugs and Crime.
\newblock {\em World Drug Report 2025}.
\newblock United Nations, Vienna, 2025.

\bibitem{WHO2024}
World~Health Organization.
\newblock {\em Global Status Report on Alcohol and Health and Treatment of Substance Use Disorders}.
\newblock World Health Organization, Geneva, 2024.

\bibitem{Petzold1983}
L.~R. Petzold.
\newblock Automatic selection of methods for solving stiff and nonstiff systems of ordinary differential equations.
\newblock {\em SIAM J. Sci. Stat. Comput.}, 4(1):136--148, 1983.

\bibitem{reluga2010}
T.~C. Reluga.
\newblock Game theory of social distancing in response to an epidemic.
\newblock {\em PLoS Comput. Biol.}, 6(5):e1000793, 2010.

\bibitem{Rowe1992}
D.~C. Rowe, L.~Chassin, C.~C. Presson, D.~Edwards, and S.~J. Sherman.
\newblock An ``epidemic'' model of adolescent cigarette smoking.
\newblock {\em J. Appl. Soc. Psychol.}, 22(4):261--285, 1992.

\bibitem{sanchez2023}
F.~Sanchez, J.~Arroyo-Esquivel, and J.~G. Calvo.
\newblock A mathematical model with nonlinear relapse: conditions for a forward-backward bifurcation.
\newblock {\em J. Biol. Dyn.}, 17(1):2192238, 2023.

\bibitem{sanchez2007}
F.~Sanchez, X.~Wang, and C.~Castillo-Chavez.
\newblock Drinking as an epidemic.
\newblock In {\em Therapist's Guide to Evidence Based Relapse Prevention}, pages 353--368. Academic Press, 2007.

\bibitem{Toxvaerd2020}
F.~Toxvaerd.
\newblock Equilibrium social distancing.
\newblock Technical Report 2020/08, Cambridge-INET, 2020.

\bibitem{vandenDriesscheWatmough2002}
P.~van~den Driessche and J.~Watmough.
\newblock Reproduction numbers and sub-threshold endemic equilibria for compartmental models of disease transmission.
\newblock {\em Math. Biosci.}, 180:29--48, 2002.

\bibitem{Wilson1927}
E.~B. Wilson.
\newblock Probable inference, the law of succession, and statistical inference.
\newblock {\em J. Am. Stat. Assoc.}, 22(158):209--212, 1927.

\bibitem{Zhang2024}
T.~Zhang, L.~Sun, X.~Yin, et~al.
\newblock Burden of drug use disorders in the United States from 1990 to 2021 and its projection until 2035: results from the GBD study.
\newblock {\em BMC Public Health}, 24:1639, 2024.

\end{thebibliography}
\end{document}